\documentclass[11pt]{amsart}
\usepackage{amssymb}
\usepackage{bm}
\usepackage{graphicx}
\usepackage[centertags]{amsmath}
\usepackage{amsfonts}
\usepackage{amsthm}

\theoremstyle{plain}
\theoremstyle{definition}
\newtheorem{theorem}{Theorem}[section]
\newtheorem{lemma}[theorem]{Lemma}
\newtheorem{proposition}[theorem]{Proposition}

\newtheorem{remark}[theorem]{Remark}

\newcommand{\lh}{\widehat{L}}

\begin{document}

\title{Polynomial estimates and radius of analyticity on real Banach spaces}
\author{M. K. Papadiamantis }

\begin{abstract}
A generalization of Problem 73 of Mazur and Orlicz in the Scottish Book was introduced from L. Harris. The exact value of the constant that appears there is known when complex normed linear spaces are considered. In this paper, we give estimates in the case of an arbitrary real normed linear space and a real $\ell_p$ space. Moreover, if $F(x)$ is a power series, $\rho$ its radius of uniform convergence and $\rho_{\substack{\\ A}}$ its radius of analyticity, we prove that $\rho_{\substack{\\ A}}\geq\rho/\sqrt{2}$ and give some respective results for the $n$th Fr\'{e}chet derivative of $F(x)$.
\end{abstract}

\maketitle

\section{introduction}

If $X$ is a Banach space over $\mathbb{K}$, $\mathbb{K}=\mathbb{R}$ or $\mathbb{C}$, we let
$\mathcal{L}^s(^mX)$ denote the Banach space of all continuous symmetric $m$-linear forms $L:X^m\rightarrow\mathbb{K}$ with the norm
\[\|L\|=\sup\{|L(x_1,\ldots,x_m)|: \|x_1\| \leq 1,\ldots, \|x_m\| \leq 1\}\,.\]
A function $P: X\rightarrow\mathbb{K}$ is a continuous $m$-homogeneous polynomial if there is a continuous symmetric $m$-linear form $L:X^m\rightarrow\mathbb{K}$ for which $P(x)=L(x, \ldots, x)$ for all $x\in X$. In this case it is convenient to write $P=\lh$. Let $\mathcal{P}(^mX)$ denote the Banach space
of all continuous $m$-homogeneous polynomials $P: X\rightarrow\mathbb{K}$ with the norm
\[\|P\|=\sup\{|P(x)|: \|x\|\leq 1\}\,.\]
We write $L(x_1^{k_1}\ldots x_n^{k_n})$ as shorthand for $L(x_1,\ldots,x_1,\ldots,x_n,\ldots,x_n)$ where each $x_i$ appears $k_i$ times for $1\leq i\leq n$, $k_1+\ldots+k_n=m$.\\
We also need to define the norm
\[\|L\|_{(n)}=\sup_{k_1+\ldots k_n=m}\sup\{|L(x_1^{k_1}\ldots x_n^{k_n})|: \|x_1\| \leq 1,\ldots, \|x_n\| \leq 1\},\]
which for $n=2$ gives
\[\|L\|_{(2)}=\sup_{1\leq k\leq m}\sup\{|L(x_1^kx_2^{m-k})|: \|x_1\| \leq 1, \|x_2\| \leq 1\}\,.\]
Clearly, $\|\lh\|\leq\|L\|_{(n)}\leq\|L\|_{(n+1)}\leq\|L\|$.

It is known that if $L\in\mathcal{L}^s(^mX)$ and $\lh$ the associated polynomial, then
\[\|L\|\leq\frac{m^m}{m!}\|\lh\|.\]
This is the answer of Problem 73 of Mazur and Orlicz in \cite{SB}.

A natural generalization of Problem 73 is the following:\\
Let $X$ be a normed linear space, $k_1,\ldots,k_n$ be nonnegative integers whose sum is $m$ and let $c(k_1,\ldots,k_n,X)$ be the smallest number with the property that if $L$ is any symmetric $m$-linear mapping of one real normed linear space into another, then
\[|L(x_1^{k_1}\ldots x_n^{k_n})|\leq c(k_1,\ldots,k_n,X)\|\lh\|.\]
It is shown in \cite{H1} Theorem 1, that if only complex normed linear spaces and comlex scalars are considered, then
\[c(k_1,\ldots,k_n,X)=\frac{k_1!\ldots k_n!}{k_1^{k_1}\ldots k_n^{k_n}}\frac{m^m}{m!}.\]
In the next two sections we shall find bounds for the constant $c(k_1,\ldots,k_n,X)$ in the case of real normed linear spaces and specifically in the case of real $\ell_p$ spaces.

\section{Polynomials on a real normed linear space}

From \cite{H2} Corollary 7 (see also \cite{RS}, \cite{S2}), we have
\begin{equation} \label{eq Harris}
|L(x_1^{k_1}\ldots x_n^{k_n})|\leq\sqrt{\frac{m^m}{k_1^{k_1}\ldots k_n^{k_n}}}\|\lh\|
\end{equation}
for all non-negative integers $k_1,\ldots,k_n$ with $k_1+\ldots+k_n=m$.\\
For $n=2$ we get
\[|L(x_1^kx_2^{m-k})|\leq\sqrt{\frac{m^m}{k^k(m-k)^{m-k}}}\|\lh\|\]
and we can easily see that the square root takes its maximum value which is $(\sqrt{2})^m$ when $k=\frac{m}{2}$.\\
To see this, simply consider the function $f(k)=k^k(m-k)^{m-k}$.\\
Then,
\begin{eqnarray*}
f'(k) &=& (1+\ln k)k^k(m-k)^{m-k}-k^k[1+\ln(m-k)](m-k)^{m-k}\\
&=& k^k(m-k)^{m-k}\ln\frac{k}{m-k}
\end{eqnarray*}
and
\[f'(k)=0\Leftrightarrow\ln\frac{k}{m-k}=0\Leftrightarrow\frac{k}{m-k}=1\Leftrightarrow k=\frac{m}{2},\]
since $k\neq0$ and $k\neq m$.\\
Moreover, $f'(k)<0$ for $0<k<\frac{m}{2}$ (i.e. $f$ if strictly decreasing for $0<k<\frac{m}{2}$) and $f'(k)>0$ for $\frac{m}{2}<k<m$ (i.e. $f$ if strictly increasing for $\frac{m}{2}<k<m$).\\
Therefore, the minimum value of $f$ is $f\big(\frac{m}{2}\big)=\big(\frac{m}{2}\big)^m$. Thus,
\[\bigg(\frac{\|L\|_{(2)}}{\|\lh\|}\bigg)^{\frac{1}{m}}\leq\sqrt{2}.\]
If $m$ is odd, the previous inequality is strict (since $k\in\mathbb{N}$).\\
Similarly, we get the following general result:
\[\bigg(\frac{\|L\|_{(n)}}{\|\lh\|}\bigg)^{\frac{1}{m}}\leq\sqrt{n}\]

From \cite{N} Lemma 3, we have
\[|L(x_1^{k_1}\ldots x_n^{k_n})|\leq e^{\frac{m}{2}}C_m
\begin{pmatrix} m+n-1\\n-1 \end{pmatrix}
\|\lh\|.\]
The value of the constant $C_m$ needs to be determined in order to compare this result with (\ref{eq Harris}).
To do this, we need to find a bound for $\mathbb{E}(|A|^k)$ which is used in the proof of \cite{N} Lemma 3. Stirling's formula and some elementary calculations will do the job.\\
For $k\neq2$ even:
\begin{eqnarray*}
\mathbb{E}(|A|^k) &\leq& k(2p)^\frac{k}{2}\Gamma\bigg(\frac{k}{2}\bigg)=k(2p)^\frac{k}{2}\bigg(\frac{k}{2}-1\bigg)!=
k(2p)^\frac{k}{2}\bigg(\frac{\frac{k}{2}-1}{e}\bigg)^{\frac{k}{2}-1}\sqrt{\frac{k}{2}-1}\\
&=& k(2p)^\frac{k}{2}\bigg(\frac{k-2}{2e}\bigg)^\frac{k}{2}\frac{2e}{k-2}\sqrt{\frac{k-2}{2}}=
\frac{ke\sqrt{2}}{\sqrt{k-2}}\bigg[\frac{p(k-2)}{e}\bigg]^\frac{k}{2}\\
&=& \frac{ke\sqrt{2}}{\sqrt{k-2}}\bigg(\frac{k-2}{k}\bigg)^\frac{k}{2}\bigg(\frac{pk}{e}\bigg)^\frac{k}{2}=
e\sqrt{2}\sqrt{k-2}\bigg(\frac{k-2}{k}\bigg)^{\frac{k}{2}-1}\bigg(\frac{pk}{e}\bigg)^\frac{k}{2}\\
&\leq& e\sqrt{2(k-2)}\bigg(\frac{pk}{e}\bigg)^\frac{k}{2}.
\end{eqnarray*}
For $k\neq1$ odd:
\begin{eqnarray*}
\mathbb{E}(|A|^k) &\leq& k(2p)^\frac{k}{2}\Gamma\bigg(\frac{k}{2}\bigg)\leq k(2p)^\frac{k}{2}\Gamma\bigg(\frac{k+1}{2}\bigg)=k(2p)^\frac{k}{2}\bigg(\frac{k-1}{2}\bigg)!\\
&=& k(2p)^\frac{k}{2}\bigg(\frac{\frac{k-1}{2}}{e}\bigg)^{\frac{k-1}{2}}\sqrt{\frac{k-1}{2}}=
k(2p)^\frac{k}{2}\bigg(\frac{k-1}{2e}\bigg)^\frac{k}{2}\sqrt{\frac{2e}{k-1}}\sqrt{\frac{k-1}{2}}\\
&=& k\sqrt{e}\bigg[\frac{p(k-1)}{e}\bigg]^\frac{k}{2}=
k\sqrt{e}\bigg(\frac{k-1}{k}\bigg)^\frac{k}{2}\bigg(\frac{pk}{e}\bigg)^\frac{k}{2}\\
&=& (k-1)\sqrt{e}\bigg(\frac{k-1}{k}\bigg)^{\frac{k}{2}-1}\bigg(\frac{pk}{e}\bigg)^\frac{k}{2}\leq
(k-1)\sqrt{e}\bigg(\frac{pk}{e}\bigg)^\frac{k}{2}.
\end{eqnarray*}
Moreover, $\mathbb{E}(|A|)\leq\sqrt{2p}$ and $\mathbb{E}(|A|^2)\leq4p$.\\
In order to make the next calculations easier, we have to unite all the above cases. Thus, we take $\mathbb{E}(|A|^k)\leq ke\big(\frac{pk}{e}\big)^\frac{k}{2}$. The term $ke$ could be slightly better, but it would not make any essential difference.\\
At the end of the proof of \cite{N} Lemma 3, we need to calculate the $\sup_{k_i}(C_{k_1}\ldots C_{k_n})$, where $k_1+\ldots+k_n=m$. We have:
\[\sup_{k_i}(C_{k_1}\ldots C_{k_n})=\sup_{k_i}(k_1e\ldots k_ne)=e^n\sup_{k_i}(k_1\ldots k_n)\leq e^n\bigg(\frac{m}{n}\bigg)^n=\bigg(\frac{em}{n}\bigg)^n\]
Finally, we get
\begin{equation} \label{eq Nguyen}
|L(x_1^{k_1}\ldots x_n^{k_n})|\leq e^{\frac{m}{2}}\bigg(\frac{em}{n}\bigg)^n
\begin{pmatrix} m+n-1\\n-1 \end{pmatrix}
\|\lh\|.
\end{equation}
Inequality (\ref{eq Nguyen}) is worse than inequality (\ref{eq Harris}). But asymptotically it gives better estimate, so it can be useful in some cases.

Combining inequalities (\ref{eq Harris}) and (\ref{eq Nguyen}) we obtain the following

\begin{lemma} \label{lemma asymptotic}
Let $L:X^m\rightarrow Y$ be an $m$-linear map. Then
\[
\bigg(\frac{\|L\|_{(n)}}{\|\lh\|}\bigg)^{\frac{1}{m}}\leq
\begin{cases}
\sqrt{2},\ \text{if }n=2\\
C\sqrt{e},\ \text{if }n\geq3
\end{cases}
\]
where $C=C(m,n)$ is independent of $L$, $X$, $Y$ and tends to 1 as $m\rightarrow\infty$ for fixed $n$.
\end{lemma}

The $n$th Rademacher function is defined on $[0,1]$ by $r_n=\text{sign}\sin2^n\pi t$.

\begin{lemma}({\bf{Polarization Formula}}) \label{polarization formula}
Let $X$ be a vector space and $L\in\mathcal{L}^s(^mX)$. If $x_1,\ldots,x_m\in X$, then
\begin{equation} \label{eq polarization formula}
L(x_1,\ldots,x_m) = \frac{1}{m!}\int_0^1r_1(t)\ldots r_m(t)\lh\bigg[\sum_{i=1}^mr_i(t)x_i\bigg]dt.
\end{equation}
\end{lemma}

The proof of Lemma \ref{polarization formula} is easy, so it is omitted (see \cite{S1} Lemma 2).

\begin{proposition} \label{proposition X}
Let $X$, $Y$ be real normed linear spaces and $L:X^m\rightarrow Y$ be a continuous symmetric $m$-linear mapping with associated homogeneous polynomial $\lh$. If $x_1,\ldots,x_n$ are norm-one vectors in $X$, then
\[\frac{|L(x_1^{k_1}\ldots x_n^{k_n})|}{\|\lh\|}\leq\min\bigg\{
\sqrt{\frac{m^m}{k_1^{k_1}\ldots k_n^{k_n}}}\ ,
\frac{k_1^{k_1}\ldots k_n^{k_n}}{m!}n^m\bigg\}\]
for all non-negative integers $k_1,\ldots,k_n$ with $k_1+\ldots+k_n=m$.
\end{proposition}

\begin{proof}
Using Lemma \ref{polarization formula}, we have
\[L(x_1^{k_1}\ldots x_n^{k_n}) = \frac{1}{m!}\int_0^1r_1(t)\ldots r_m(t)\lh[(r_1(t)+\ldots+r_{k_1}(t))x_1+\ldots]dt\]

\[\Leftrightarrow L\bigg(\bigg(\frac{x_1}{k_1}\bigg)^{k_1}\ldots\bigg(\frac{x_n}{k_n}\bigg)^{k_n}\bigg) =
\frac{1}{m!}\int_0^1r_1(t)\ldots r_m(t)\lh\bigg[(r_1(t)+\ldots+r_{k_1}(t))\frac{x_1}{k_1}+\ldots\bigg]dt\]

\[\Leftrightarrow \frac{1}{k_1^{k_1}\ldots k_n^{k_n}}L(x_1^{k_1}\ldots x_n^{k_n}) =
\frac{1}{m!}\int_0^1r_1(t)\ldots r_m(t)\lh\bigg[(r_1(t)+\ldots+r_{k_1}(t))\frac{x_1}{k_1}+\ldots\bigg]dt\]
Therefore,
\begin{equation} \label{eq new}
\begin{aligned}
|L(x_1^{k_1}\ldots x_n^{k_n})| & \leq
\frac{k_1^{k_1}\ldots k_n^{k_n}}{m!}\|\lh\|\int_0^1\bigg\|(r_1(t)+\ldots+r_{k_1}(t))\frac{x_1}{k_1}+\ldots\bigg\|^mdt\\
& \leq
\frac{k_1^{k_1}\ldots k_n^{k_n}}{m!}\|\lh\|\int_0^1\bigg(\frac{|r_1(t)+\ldots+r_{k_1}(t)|}{k_1}+\ldots\bigg)^mdt\\
& \leq \frac{k_1^{k_1}\ldots k_n^{k_n}}{m!}n^m\|\lh\|.
\end{aligned}
\end{equation}
Inequalities (\ref{eq Harris}) and (\ref{eq new}) complete the proof.
\end{proof}

Asymptotically, inequality (\ref{eq new}) gives much worse estimate than inequalities (\ref{eq Harris}) and (\ref{eq Nguyen}). It is useful though for large $n$'s and when $n$ depends on $m$, while inequality (\ref{eq Harris}) is useful only for fixed $n$'s. We now conclude that:

\begin{theorem} \label{theorem X}
Let $X$, $Y$ be real normed linear spaces and $L:X^m\rightarrow Y$ be a continuous symmetric $m$-linear mapping with associated homogeneous polynomial $\lh$. If $x_1,\ldots,x_n$ are norm-one vectors in $X$, then
\[\frac{k_1!\ldots k_n!}{k_1^{k_1}\ldots k_n^{k_n}}\frac{m^m}{m!}\leq c(k_1,\ldots,k_n,X)\leq\min\bigg\{
\sqrt{\frac{m^m}{k_1^{k_1}\ldots k_n^{k_n}}}\ ,
\frac{k_1^{k_1}\ldots k_n^{k_n}}{m!}n^m\bigg\}\]
for all non-negative integers $k_1,\ldots,k_n$ with $k_1+\ldots+k_n=m$.
\end{theorem}

\begin{proof}
The right hand side inequality is an immediate consequence of Proposition \ref{proposition X}.\\
For the left hand side, let $x^i=(x_n^i)_{n=1}^\infty\in X$, $i=1,\ldots,m$ and $l\in\mathcal{L}^s(^mX)$ be defined by
\[L(x^1,\ldots,x^m)=\frac{1}{m!}\sum_{\sigma\in S_m}x_{\sigma(1)}^1\ldots x_{\sigma(m)}^m,\]
where $S_m$ is the set of permutations of the first $m$ natural numbers.
Then
\[L\big((x^1)^{k_1}\ldots(x^n)^{k_n}\big)=\frac{1}{m!}\sum_{\sigma\in S_m}x_{\sigma(1)}^1\ldots x_{\sigma(k_1)}^1\ldots x_{\sigma(k_1+\ldots k_{n-1}+1)}^n\ldots x_{\sigma(k_1+\ldots k_n)}^n.\]
Take $e_i$ to be the $i$th coordinate vector of $X$ and define
\[y^1=\frac{1}{k_1}(e^1+\ldots+e^{k_1})\]
\[y^2=\frac{1}{k_2}(e^{k_1+1}+\ldots+e^{k_1+k_2})\]
\[\vdots\]
\[y^n=\frac{1}{k_n}(e^{k_1+\ldots+k_{n-1}+1}+\ldots+e^{k_1+\ldots+k_n}).\]
Then an easy calculation shows that $y^1,\ldots,y^n$ are unit vectors in $X$ and
\[L\big((y^1)^{k_1}\ldots(y^n)^{k_n}\big)=\frac{1}{m!}\frac{k_1!\ldots k_n!}{k_1^{k_1}\ldots k_n^{k_n}}.\]
On the other hand, $\|\lh\|\leq\frac{1}{m^m}$, since
\[|\lh(x)|=|x_1\ldots x_m|=\big[(|x_1|\ldots|x_m|)^\frac{1}{m}\big]^m\leq\bigg(\frac{|x_1|+\ldots+|x_m|}{m}\bigg)^m\]
by the arithmetic-geometric mean inequality. Thus
\[\big|L\big((y^1)^{k_1}\ldots(y^n)^{k_n}\big)\big|\geq\frac{k_1!\ldots k_n!}{k_1^{k_1}\ldots k_n^{k_n}}\frac{m^m}{m!}\|\lh\|.\]
\end{proof}

\begin{remark}
For $k_1=\ldots=k_n=1$, the upper and lower bound of $c(k_1,\ldots,k_n,X)$ in Theorem \ref{theorem X} give the same estimate which is $\frac{m^m}{m!}$.
\end{remark}

\section{polynomials on a real $\ell_p$ space}

In \cite{S1} Theorem 2, Y. Sarantopoulos proved that in the case of $L_p(\mu)$, for $1\leq p\leq m'$, $\frac{1}{m}+\frac{1}{m'}=1$ holds that
\[\|L\|\leq\frac{m^\frac{m}{p}}{m!}\|\lh\|.\]
This is an improved estimate of the one in \cite{H1} by L. A. Harris.
Here we shall give some estimates of the constant $c(k_1,\ldots,k_n,\ell_p)$ for which $|L(x_1^{k_1}\ldots x_n^{k_n})|\leq c(k_1,\ldots,k_n,\ell_p)\|\lh\|$.

If $f$ is a measurable function on $(X,\mathcal{A},\mu)$, we define its distribution function $\lambda_f:(0,+\infty)\rightarrow[0,+\infty]$ by
\[\lambda_f(a)=\mu(\{x:|f(x)|>a\}).\]
From \cite{F}, we have the following

\begin{proposition}
If $\lambda_f(a)<\infty$ for every $a>0$ and $\phi$ is a nonnegative Borel function on $(0,\infty)$, then
\[\int_X\phi\circ|f|d\mu=-\int_0^\infty\phi(a)d\lambda_f(a).\]
\end{proposition}

The case of this result we are interested in, is $\phi(a)=a^p$, which gives
\[\int|f|^pd\mu=-\int_0^{\infty}a^pd\lambda_f(a).\]
Integrating the right side by parts, we obtain
\[\int|f|^pd\mu=p\int_0^{\infty}a^{p-1}\lambda_f(a)da.\]
The validity of this calculation becomes clear if we consider that $a^p\lambda_f(a)\rightarrow0$ as $a\rightarrow0$ and $a\rightarrow\infty$ (since $\lambda_f$ is strictly decreasing). In the following Proposition the function $f$ will be of the form $f(t)=r_1(t)+\ldots+r_k(t)$, $k\in\mathbb{N}$. Therefore, using Hoeffding's inequality (see \cite{H} Theorem 2), we get that
\[\lambda_f(x):=\lambda_k(x)=P(|r_1(t)+\ldots+r_k(t)|\geq x)\leq2e^{-\frac{x^2}{2k}}.\]

\begin{proposition} \label{proposition l_p}
Let $1\leq p\leq\infty$ and $L:(\ell_p)^m\rightarrow\mathbb{R}$ be a continuous symmetric $m$-linear mapping with associated homogeneous polynomial $\lh$. If $x_1,\ldots,x_n$ are norm-one vectors in $\ell_p$ with disjoint supports, then
\[\frac{|L(x_1^{k_1}\ldots x_n^{k_n})|}{\|\lh\|}\leq
\begin{cases}
\min\bigg\{\frac{k_1^{k_1}\ldots k_n^{k_n}}{m!}n^\frac{m}{p},
\frac{p2^\frac{p}{2}\Gamma\big(\frac{p}{2}\big)m^\frac{p}{2}}{m!}\bigg\},\text{ if }p\geq m\\
\min\bigg\{\frac{k_1^{k_1}\ldots k_n^{k_n}}{m!}n^\frac{m}{p},
\frac{n^\frac{m-p}{p}m2^\frac{m}{2}\Gamma\big(\frac{m}{2}\big)m^\frac{m}{2}}{m!}\bigg\},\text{ if } p<m
\end{cases}\]
for all non-negative integers $k_1,\ldots,k_n$ with $k_1+\ldots+k_n=m$.
\end{proposition}

\begin{proof}
Working as in the proof of Proposition \ref{proposition X}, we get
\begin{equation} \label{eq l_p}
\begin{aligned}
|L(x_1^{k_1}\ldots x_n^{k_n})| & \leq
\frac{k_1^{k_1}\ldots k_n^{k_n}}{m!}\|\lh\|\int_0^1\bigg\|(r_1(t)+\ldots+r_{k_1}(t))\frac{x_1}{k_1}+\ldots\bigg\|_p^mdt\\
& \leq
\frac{k_1^{k_1}\ldots k_n^{k_n}}{m!}\|\lh\|\int_0^1\bigg[\bigg(\frac{|r_1(t)+\ldots+r_{k_1}(t)|}{k_1}\bigg)^p+\ldots\bigg]^\frac{m}{p}dt\\
& \leq \frac{k_1^{k_1}\ldots k_n^{k_n}}{m!}n^\frac{m}{p}\|\lh\|.
\end{aligned}
\end{equation}

We shall now use a different technique and we need to distinguish two cases.

For $p\geq m$:
\begin{eqnarray*}
|L(x_1^{k_1}\ldots x_n^{k_n})| &\leq&
\frac{\|\lh\|}{m!}\int_0^1\|(r_1(t)+\ldots+r_{k_1}(t))x_1+\ldots\|_p^mdt\\
&\leq& \frac{\|\lh\|}{m!}\int_0^1(|r_1(t)+\ldots+r_{k_1}(t)|^p+\ldots)^{\frac{m}{p}}dt\\
&\leq& \frac{\|\lh\|}{m!}\bigg(p\int_0^{\infty}x^{p-1}\lambda_{k_1}(x)dx+\ldots+p\int_0^{\infty}x^{p-1}\lambda_{k_n}(x)dx\bigg)\\
&\leq& \frac{\|\lh\|}{m!}\bigg(p\int_0^{\infty}x^{p-1}2e^{-\frac{x^2}{2k_1}}dx+\ldots+p\int_0^{\infty}x^{p-1}2e^{-\frac{x^2}{2k_n}}dx\bigg)\\
&=& \frac{\|\lh\|}{m!}\bigg(p(2k_1)^{\frac{p}{2}}\Gamma\big(\frac{p}{2}\big)+\ldots+p(2k_n)^{\frac{p}{2}}\Gamma\big(\frac{p}{2}\big)\bigg)\\
&=& \frac{p2^\frac{p}{2}\Gamma\big(\frac{p}{2}\big)}{m!}\sum_{i=1}^nk_i^\frac{p}{2}\|\lh\|\leq
\frac{p2^\frac{p}{2}\Gamma\big(\frac{p}{2}\big)}{m!}\bigg(\sum_{i=1}^nk_i\bigg)^\frac{p}{2}\|\lh\|\\
&=& \frac{p2^\frac{p}{2}\Gamma\big(\frac{p}{2}\big)m^\frac{p}{2}}{m!}\|\lh\|
\end{eqnarray*}

For $p<m$:
\begin{eqnarray*}
|L(x_1^{k_1}\ldots x_n^{k_n})| &\leq&
\frac{\|\lh\|}{m!}\int_0^1\|(r_1(t)+\ldots+r_{k_1}(t))x_1+\ldots\|_p^mdt\\
&\leq& \frac{\|\lh\|}{m!}\int_0^1(|r_1(t)+\ldots+r_{k_1}(t)|^p+\ldots)^{\frac{m}{p}}dt\\
&=& \frac{n^\frac{m}{p}\|\lh\|}{m!}\int_0^1\bigg(\frac{|r_1(t)+\ldots+r_{k_1}(t)|^p+\ldots}{n}\bigg)^{\frac{m}{p}}dt\\
&\leq& \frac{n^\frac{m}{p}\|\lh\|}{m!}\int_0^1\frac{|r_1(t)+\ldots+r_{k_1}(t)|^m+\ldots}{n}dt\\
&\leq& \frac{n^\frac{m-p}{p}\|\lh\|}{m!}\bigg(m\int_0^{\infty}x^{m-1}\lambda_{k_1}(x)dx+\ldots+m\int_0^{\infty}x^{m-1}\lambda_{k_n}(x)dx\bigg)\\
&\leq& \frac{n^\frac{m-p}{p}\|\lh\|}{m!}\bigg(m\int_0^{\infty}x^{m-1}2e^{-\frac{x^2}{2k_1}}dx+\ldots+m\int_0^{\infty}x^{m-1}2e^{-\frac{x^2}{2k_n}}dx\bigg)\\ &=& \frac{n^\frac{m-p}{p}\|\lh\|}{m!}\bigg(m(2k_1)^{\frac{m}{2}}\Gamma\big(\frac{m}{2}\big)+\ldots+m(2k_n)^{\frac{m}{2}}\Gamma\big(\frac{m}{2}\big)\bigg)\\
&=& \frac{n^\frac{m-p}{p}m2^\frac{m}{2}\Gamma\big(\frac{m}{2}\big)}{m!}\sum_{i=1}^nk_i^\frac{m}{2}\|\lh\|\leq
\frac{n^\frac{m-p}{p}m2^\frac{m}{2}\Gamma\big(\frac{m}{2}\big)}{m!}\bigg(\sum_{i=1}^nk_i\bigg)^\frac{m}{2}\|\lh\|\\
&=& \frac{n^\frac{m-p}{p}m2^\frac{m}{2}\Gamma\big(\frac{m}{2}\big)m^\frac{m}{2}}{m!}\|\lh\|
\end{eqnarray*}
\end{proof}

Note that the second technique used in the proof of Proposition \ref{proposition l_p} gives better estimates than (\ref{eq l_p}) for large $m$'s and small $n$'s and its special case where $p\geq m$ holds asymptotically only for $\ell_\infty$. Moreover, observe that Proposition \ref{proposition l_p} gives better estimates than Proposition \ref{proposition X}.

\begin{remark}
The term $\sum_{i=1}^nk_i^\frac{p}{2}$ \big(resp. $\sum_{i=1}^nk_i^\frac{m}{2}$\big) in the proof of Proposition \ref{proposition l_p} can be bounded by $(m-n+1)^\frac{p}{2}+n-1$ \big(resp. $(m-n+1)^\frac{m}{2}+n-1$\big) instead of $m^\frac{p}{2}$ \big(resp. $m^\frac{m}{2}$\big) since it takes its maximum value when all but one $k_i$'s equal to 1 and the last one equals to $m-n+1$.
\end{remark}

We now conclude that:

\begin{theorem} \label{theorem l_p}
Let $1\leq p\leq\infty$ and $L:(\ell_p)^m\rightarrow\mathbb{R}$ be a continuous symmetric $m$-linear mapping with associated homogeneous polynomial $\lh$. If $x_1,\ldots,x_n$ are norm-one vectors in $\ell_p$ with disjoint supports, then:\\
if $p\geq m$
\[\frac{k_1!\ldots k_n!}{k_1^\frac{k_1}{p}\ldots k_n^\frac{k_n}{p}}\frac{m^\frac{m}{p}}{m!}\leq c(k_1,\ldots,k_n,\ell_p)\leq\min\bigg\{\frac{k_1^{k_1}\ldots k_n^{k_n}}{m!}n^\frac{m}{p},
\frac{p2^\frac{p}{2}\Gamma\big(\frac{p}{2}\big)m^\frac{p}{2}}{m!}\bigg\}\]
and if $p<m$
\[\frac{k_1!\ldots k_n!}{k_1^\frac{k_1}{p}\ldots k_n^\frac{k_n}{p}}\frac{m^\frac{m}{p}}{m!}\leq c(k_1,\ldots,k_n,\ell_p)\leq\min\bigg\{\frac{k_1^{k_1}\ldots k_n^{k_n}}{m!}n^\frac{m}{p},
\frac{n^\frac{m-p}{p}m2^\frac{m}{2}\Gamma\big(\frac{m}{2}\big)m^\frac{m}{2}}{m!}\bigg\}\]
for all non-negative integers $k_1,\ldots,k_n$ with $k_1+\ldots+k_n=m$.
\end{theorem}

\begin{proof}
The right hand side inequality is an immediate consequence of Proposition \ref{proposition l_p}.\\
For the left hand side, we just need some simple adjustments to the proof of Proposition \ref{proposition X}.
\end{proof}

\begin{remark}
For $k_1=\ldots=k_n=1$, the upper and lower bound of $c(k_1,\ldots,k_n,\ell_p)$ in Theorem \ref{theorem l_p} give the same estimate which is $\frac{m^\frac{m}{p}}{m!}$.
\end{remark}

\section{Radius of analyticity of a power series on a real Banach space}

A power series centered at $a\in X$ is a formal sum
\begin{equation} \label{eq power series}
\sum_{m=0}^\infty P_m(x-a)
\end{equation}
where for each $m$, $P_m:X\rightarrow Y$ is a continuous $m$-homogeneous polynomial.

The radius of uniform convergence of this power series is defined to be
\[\rho:=\{r:(\ref{eq power series})\text{ converges uniformly on }|x-a|\leq r\},\]
which is given by the following standard formula:
\[\rho=\frac{1}{\limsup_{m\rightarrow\infty}\|P_m\|^\frac{1}{m}}.\]
If $\rho>0$, then for every $0<r<\rho$, (\ref{eq power series}) is a uniformly and absolutely convergent series for every $x\in B_r(a)$. hence in this case, (\ref{eq power series}) defines a function on $B_{\rho}(a)$ taking values in $Y$.

The Taylor series of an infinitely differentiable function $F$ defined in a neighborhood of $a$ is the power series defined by
\[T_aF(x)=\sum_{m=0}^\infty\frac{1}{m!}D^mF(a)((x-a)^m),\]
where $D^mF(a):X^m\rightarrow Y$ is the symmetric $m$-linear map given by taking the Fr\'{e}chet derivative of $F$ $m$ times.

$F$ is called analytic at $a$ if $T_aF(x)$ has a positive radius of uniform convergence and equals $F(x)$ within the domain of uniform convergence. If $U\subset X$ is open, we say that $F$ is analytic in $U$ if it is analytic at every $a\in U$. If furthermore, $T_aF(x)$ converges uniformly in every closed ball centered at $a$ contained in $U$, for each $a\in U$, $F$ is called fully analytic in $U$.

Let $F(x)$ be a power series centered at $a$ with radius of uniform convergence $\rho>0$. The radius of analyticity $\rho_{\substack{\\ A}}=\rho_{\substack{\\ A}}(F)$ of $F(x)$ at $a$ is the largest $r>0$ such that $F(x)$ is fully analytic in $B_r(a)$.

The norm we need to control when we expand a power series at a new point is $\|L\|_{(2)}$. To see this, consider the power series (\ref{eq power series}) centered at $a=0$, which we may rewrite as
\begin{equation} \label{eq series at 0}
F(x)=\sum_{m=0}^\infty L_m(x^m).
\end{equation}
Observe that by the binomial formula, given any $y\in X$
\[L_m(x^m)=L_m((y+x-y)^m)=\sum_{k=0}^\infty
\begin{pmatrix} m\\k \end{pmatrix}
L_m(y^{m-k},(x-y)^k).\]
Then, if in
\begin{equation} \label{eq double series}
\sum_{m=0}^\infty L_m(x^m)=\sum_{m=0}^\infty\sum_{k=0}^m
\begin{pmatrix} m\\k \end{pmatrix}
L_m(y^{m-k},(x-y)^k)
\end{equation}
the double series on the right converges absolutely, we can interchange summations and obtain
\[F(x)=\sum_{m=0}^\infty L_m(x^m)=\sum_{k=0}^\infty\sum_{m=k}^\infty
\begin{pmatrix} m\\k \end{pmatrix}
L_m(y^{m-k},(x-y)^k)\]
Thus, if we can perform this change of summation for all $x\in B_r(y)$, for some $r>0$, then we will have expressed $F(x)$ as a power series centered at $y$ whose $k$-homogeneous polynomial coefficients are given by
\begin{equation} \label{eq coefficients}
A_k(z):=\sum_{m=k}^\infty
\begin{pmatrix} m\\k \end{pmatrix}
L_m(y^{m-k},z^k).
\end{equation}
Observe that the absolute convergence of the double sum (\ref{eq double series}) for $x\in B_r(y)$ implies the absolute convergence of the $A_k$ in $B_r(0)$ and hence on all $X$ by homogeneity.

Absolute convergence of (\ref{eq double series}) holds if
\[\sum_{m=0}^\infty\sum_{k=0}^m
\begin{pmatrix} m\\k \end{pmatrix}
\|L_m\|_{(2)}|y|^{m-k}|x-y|^k=\sum_{m=0}^\infty\|L_m\|_{(2)}(|y|+|x-y|)^m<\infty.\]
This holds when
\begin{equation} \label{eq convergence1}
|y|+|x-y|<\frac{1}{\limsup\|L_m\|_{(2)}^\frac{1}{m}}.
\end{equation}
Choose a subsequence $m_i$ such that
\[\lim_{i\rightarrow\infty}\|L_{m_i}\|_{(2)}^\frac{1}{m_i}=\limsup\|L_m\|_{(2)}^\frac{1}{m}.\]
Let $\rho=\frac{1}{\limsup\|\lh_m\|_{(2)}^\frac{1}{m}}>0$ be the radius of uniform convergence of (\ref{eq series at 0}) and suppose $\rho<\infty$. Then (\ref{eq convergence1}) is satisfied if
\begin{equation}
\begin{aligned} \label{eq convergence2}
|y|+|x-y| & < \rho\cdot\frac{\limsup_{i\rightarrow\infty}\|\lh_{m_i}\|^\frac{1}{m_i}}{\lim_{i\rightarrow\infty}\|L_{m_i}\|_{(2)}^\frac{1}{m_i}}\\
& = \rho\cdot\limsup_{i\rightarrow\infty}\bigg(\frac{\|\lh_{m_i}\|}{\|L_{m_i}\|_{(2)}}\bigg)^\frac{1}{m_i}\\
& := \bar{\rho}.
\end{aligned}
\end{equation}
Thus, for $|y|<\bar{\rho}$ and $|x-y|<\bar{\rho}-|y|$, the series (\ref{eq double series}) converges absolutely.

Altogether then, we have shown that for any fixed $|y|<\bar{\rho}$, we have
\begin{equation} \label{eq series}
F(x)= \sum_{k=0}^\infty A_k(x-y)
\end{equation}
for $|x-y|<\bar{\rho}-|y|$.

From \cite{C} Corollary 1, p. 165, we have the following

\begin{lemma} \label{lemma Chae}
For any power series $F(x)= \sum_{k=0}^\infty A_k(x-y)$ centered at $y$ with positive radius of uniform convergence, $A_k=\frac{1}{k!}D^kF(y)$ as $k$-homogeneous polynomials.
\end{lemma}

\begin{theorem} \label{theorem radius}
Let $F(x)$ be a power series in a Banach space $X$, which we may take to be centered at the origin. Let $\rho>0$ denote its radius of uniform convergence and $\rho_{\substack{\\ A}}$ its radius of analyticity. Then
\begin{enumerate}
\item [(i)] $\rho_{\substack{\\ A}}\geq\frac{\rho}{\sqrt{2}}$,
\item [(ii)] for every $n$, the $n$th Fr\'{e}chet derivative $D^nF:X\rightarrow\mathcal{L}_n(X,Y)$ of $F(x)$, viewed as a map from $X$ to the Banach space $\mathcal{L}_n(X,Y)$ of continuous $n$-linear maps from $X$ to $Y$, has a Taylor series centered at the origin with radius of uniform convergence at least $\frac{\rho}{\sqrt{2}}$ for $n=2$ and $\frac{\rho}{\sqrt{e}}$ for $n\geq3$,
\item [(iii)] the radius of analyticity of the power series $D^nF(x)$ is at least $\frac{\rho}{\sqrt{2}}$.
\end{enumerate}
\end{theorem}

\begin{proof}
\begin{enumerate}
\item [(i)] Suppose $\rho<\infty$. Lemma \ref{lemma asymptotic} implies
    \begin{equation} \label{eq limsup}
    \limsup_{m\rightarrow\infty}\bigg(\frac{\|\lh_m\|}{\|L_m\|_{(2)}}\bigg)^\frac{1}{m}\geq\frac{1}{\sqrt{2},}
    \end{equation}
    hence $\bar{\rho}\geq\frac{\rho}{\sqrt{2}}$ by (\ref{eq convergence2}). Thus the preceding analysis shows that for $|y|<\rho$, the Taylor series (\ref{eq series}) is absolutely convergent in $\{x:|x-y|<\frac{\rho}{\sqrt{2}}-|y|\}$. From this, we get uniform convergence of (\ref{eq series}) in $\{x:|x-y|<r\}$ for every $r<\frac{\rho}{\sqrt{2}}-|y|$, since one can bound the tail of (\ref{eq series}) as follows:
    \begin{equation} \label{eq bound}
    \begin{aligned}
    \bigg|\sum_{k=N}^\infty A_k(x-y)\bigg| & \leq \sum_{k=N}^\infty\sum_{m=k}^\infty
    \begin{pmatrix} m\\k \end{pmatrix}
    |L_m(y^{m-k},(x-y)^k)|\\
    & \leq \sum_{m=N}^\infty\|L_m\|_{(2)}(|y|+|x-y|)^m
    \end{aligned}
    \end{equation}
    where (\ref{eq bound}) tends to zero uniformly in $x$ as $N\rightarrow0$ so long as $|y|+|x-y|\leq|y|+r$ is bounded away from $\bar{\rho}\geq\frac{\rho}{\sqrt{2}}$. This follows because we have shown that (\ref{eq bound}) viewed as power series in a single real variable has radius of uniform convergence at least $\bar{\rho}$. Finally, Lemma \ref{lemma Chae} implies that the power series (\ref{eq series}) is a Taylor series of $F(x)$ centered at $y$. So for $\rho<\infty$, this proves (i) since we have shown $\rho_{\substack{\\ A}}\geq\frac{\rho}{\sqrt{2}}$.

    For the remaining case $\rho=\infty$, i.e. $\limsup_{m\rightarrow\infty}\|\lh_m\|^\frac{1}{m}=0$, (\ref{eq limsup}) implies also that $\limsup_{m\rightarrow\infty}\|L_m\|_{(2)}^\frac{1}{m}=0$. From (\ref{eq convergence1}), we can apply the previous analysis for every $\bar{\rho}>0$, hence $\rho_{\substack{\\ A}}=\rho=\infty$.
\item [(ii)] By Chae, given any power series $F(x)=\sum_{m=0}^\infty L_m(x^m)$ with radius of uniform convergence $\rho>0$, for every $n$, $D^nF(x)$ has a Taylor series centered at the origin given by
    \begin{equation} \label{eq Taylor series}
    T_0D^nF(x)=n!\sum_{m=0}^\infty
    \begin{pmatrix} m+n\\n \end{pmatrix}
    L_{m+n}(x^m).
    \end{equation}
    In (\ref{eq Taylor series}), the linear maps $L_{m+n}$ are only evaluated on at most $n+1$ distinct arguments ($D^nF(x)$ takes values in $n$-linear maps). So by Lemma \ref{lemma asymptotic}, for $n\geq3$, the lower bound for the radius of uniform convergence of (\ref{eq Taylor series}) is $\frac{\rho}{\sqrt{e}}$ since
    \begin{eqnarray*}
    \limsup_{m\rightarrow\infty}\bigg\|
    \begin{pmatrix} m+n\\n \end{pmatrix}
    L_{m+n}\bigg\|_{(n+1)}^\frac{1}{m} &\leq& \bigg(\limsup_{m\rightarrow\infty}C(m,n+1)\sqrt{e}\bigg)\\
    & & \times\bigg(\limsup_{m\rightarrow\infty}\|\lh_{m+n}\|^\frac{1}{m}\bigg)\\
    &\leq& \frac{\sqrt{e}}{\rho}.
    \end{eqnarray*}
    Similarly, the lower bound for the radius of uniform convergence of $T_0DF(x)=\sum_{m=0}^\infty(m+1)L_{m+1}(x^m)$ is $\frac{\rho}{\sqrt{2}}$ since
    $\limsup_{m\rightarrow\infty}\|(m+1)L_{m+1}\|_{(2)}^\frac{1}{m}\leq\frac{\sqrt{2}}{\rho}$.\\
    By Lemma \ref{lemma Chae} and the formula (\ref{eq coefficients}) for the $A_k$, it follows that for $n\geq3$, $T_0D^nF(y)=D^nF(y)$ for all $y$ such that $|y|<\frac{\rho}{\sqrt{e}}$ and $T_0DF(y)=DF(y)$ for all $y$ such that $|y|<\frac{\rho}{\sqrt{2}}$.
\item [(iii)] We need to show that the radius of analyticity of (\ref{eq Taylor series}) is at least $\frac{\rho}{\sqrt{2}}$. For this, we need to control $\|L_{m+n}\|_{(n+2)}$, but this is precisely the $\|\cdot\|_{(2)}$-norm of $L_{m+n}$ viewed as a map from $X$ to $\mathcal{L}_n(X,Y)$. Applying Lemma \ref{lemma asymptotic} the proof is completed.
\end{enumerate}
\end{proof}

\end{document}